\documentclass[11pt]{amsart}
\usepackage{amssymb}
\usepackage{graphicx}
\usepackage[all,cmtip]{xy}
\renewcommand{\O}{\mathcal{O}}

\newcommand{\U}{\mathcal{U}}
\renewcommand{\P}{\mathbb{P}}

\newcommand{\SO}{{\mathcal{O}}}

\newcommand{\Z}{\mathbb{Z}}
\newcommand{\C}{\mathbb{C}}

\newcommand{\N}{\mathbb{N}}
\newcommand{\R}{\mathbb{R}}

\newcommand{\CP}{\mathbb{CP}}
\renewcommand{\S}{\mathbb{S}}

\newcommand{\Diff}{\operatorname{Diff}}

\newtheorem{proposition}{Proposition}[section]
\newtheorem{theorem}[proposition]{Theorem}
\newtheorem{definition}[proposition]{Definition}
\newtheorem{lemma}[proposition]{Lemma}

\newtheorem{remark}[proposition]{Remark}

\begin{document}

\title{Contact Blow--up}



\author{Roger Casals}
\address{Instituto de Ciencias Matem\'aticas -- CSIC.
C. Nicol\'as Cabrera, 13--15, 28049, Madrid, Spain.}
\email{casals.roger@icmat.es}

\author{Dishant M. Pancholi}
\address{Chennai Mathematical Institute,
H1 SIPCOT IT Park, Siruseri, Kelambakkam 603103, India.}
\email{dishant@cmi.ac.in}

\author{Francisco Presas}
\address{Instituto de Ciencias Matem\'aticas -- CSIC.
C. Nicol\'as Cabrera, 13--15, 28049, Madrid, Spain.}
\email{fpresas@icmat.es}



\begin{abstract}
We introduce definitions of contact blow--up from several perspectives. Such different approaches to the contact blow--up are related. We prove that the contact topology coincides in the case of blow--ups along transverse embedded loops.
\end{abstract}

\maketitle
\section{Introduction}\label{sec:intro}
In his book {\it Partial Differential Relations} ~\cite{Gr}, M. Gromov proposed a definition of the blow--up operation in the contact category, see Exercise (c) on page 343. This article discusses this definition as well as related constructions.\\

\noindent Let $M$ be a smooth manifold and $S\stackrel{e}{\hookrightarrow}M$ an embedded submanifold. The normal bundle of $(S,e)$ in $M$ will be denoted by $\nu_M(S)$. Recall that it is defined through the short exact sequence of smooth vector bundles over $S$
$$0\longrightarrow TS\stackrel{e_*}{\longrightarrow} TM|_S\longrightarrow\nu_M(S)\longrightarrow 0.$$
Given a complex vector bundle $E\longrightarrow M$, we denote by $\P(E)$ the fiberwise projectivization of $E$.\\

\noindent Suppose the normal bundle $\nu_M(S)$ is a complex bundle, then we may produce a manifold $\widetilde{M}$, the topological blow--up of $M$ along $S$. It is defined as the connected sum
$$\widetilde{M}:=M\#_S\overline{\P(\nu_M(S)\oplus \C)}$$
of the manifolds $M$ and $\P(\nu_M(S)\oplus \C)$ with the reversed orientation along $S$. Let $\sigma_0$ be the zero section of $\nu_M(S)$. The submanifold $S$ is embedded in the first factor through $e$ and in the second as the section
\begin{eqnarray*}
s: S & \longrightarrow & \P(\nu_M(S)\oplus \C) \\
p & \longmapsto & \langle(\sigma_0\oplus 1) \rangle.
\end{eqnarray*}
In the category of symplectic manifolds the normal bundle is a complex bundle and the manifold $\widetilde{M}$ can be endowed with a symplectic structure. In this paper we address the question for contact manifolds.\\

\noindent In the above reference, M. Gromov conjectured that a contact blow--up construction exists along a contact submanifold $S$ embedded in a contact manifold $M$ provided a pair of hypotheses are satisfied. These are:\\

\begin{itemize}
\item[H1.] The contact submanifold $(S, \alpha_S=e^*(\alpha))$ is a Boothby--Wang manifold. See Definition \ref{def:bw}. In particular, the Reeb vector field associated to $\alpha_S$ has all its orbits periodic with the same period. Let $W$ be the quotient space of its orbits and $\pi: S \longrightarrow W$ the projection map.\\

\item[H2.] The normal symplectic bundle $\nu_M(S)$ is isomorphic to the pull--back of a symplectic bundle $V \longrightarrow W$ through $\pi$. That is, there exists an isomorphism $\nu_M(S)\cong\pi^*V$ of symplectic bundles.
\end{itemize}

\noindent These two hypotheses allow to give a definition. Nevertheless the contact blow--up will not be a contact structure on the topological blow--up $\widetilde{M}$ of $M$. We will first illustrate a reason for this in a simple example, see Section \ref{sec:pr}. We provide a definition producing a contact structure on a manifold constructed as a different connected sum with $M$. It has the same geometrical properties as the symplectic blow--up. It is this manifold we would rather call the contact blow--up.\\

\noindent Apart from the construction of contact manifolds, the contact blow--up construction is relevant for the existence problem of contact structures on $5$--manifolds. See \cite{CPP}.\\

\noindent The content of the paper is organized as follows. The Section \ref{sec:pr} provides a brief review of the topological blow--up. In Section \ref{sec:bw}, we introduce the classical Boothby--Wang construction \cite{BW}. It will be described with some concrete examples that shall be used later on. Then, three alternative constructions of contact blow--up are introduced:\\

\begin{itemize}
\item[1.] The contact blow--up for embedded transverse loops, produced as a surgery operation. This had been introduced in the article \cite{CPP}, it will be reviewed in Section \ref{sec:srgy}.
\item[2.] The contact blow--up defined {\it \`a la Gromov} is the content of Section \ref{sec:Gromov}.
\item[3.] The contact blow--up as a contact quotient is described in Section \ref{sec:quotient}.\\
\end{itemize}
These three constructions are inspired by the three alternative constructions for the symplectic blow--up: the {\it ad hoc} construction with explicit gluings, the description using frame bundles, found in pages 239 and 243 in ~\cite{MS} respectively, and the symplectic cut procedure discussed in ~\cite{Le2}. Finally, Section \ref{sec:un} relates these constructions in the case of transverse loops.\\

\noindent {\bf Acknowledgements.} We want to acknowledge K. Niederkr\"uger for useful discussions. In particular for asking us to relate the contact cut and the contact blow--up. This project was partially developed during the AIM Workshop {\it Contact Topology in Higher Dimensions}. The first and third authors are supported by the Spanish National Research Project MTM2010--17389. Second author would like to thank ICTP for offering a visiting position that allowed him to develop this article.


\section{Preliminaries}\label{sec:pr}
\noindent In this section we introduce the basic definitions, explain the topological blow--up procedure and discuss an example. 
\begin{definition}
A contact structure on a smooth manifold $M^{2n+1}$ is a maximally non--integrable smooth field $\xi$ of tangent hyperplanes.
\end{definition}

\noindent A contact manifold $(M,\xi)$ is a choice of a contact structure $\xi$ on $M$. The maximal non--integrability can be described in terms of local equations for $\xi$. A smooth field $\xi$ of tangent hyperplanes is maximally non--integrable if and only if for any $p\in M$ there exist an open subset $U\subset M$  containing $p$ and a $1$--form  $\alpha\in\Omega^1(U)$ such that $\xi|_U=\ker\alpha$ and $\alpha\wedge d\alpha^n \neq 0$. Equivalently, the form $d \alpha$ is non--degenerate when restricted to $\xi.$ In case the form $\alpha$ can be chosen to be globally defined, i.e.  $\alpha\in\Omega^1(M)$, the contact structure $\xi$ is called coorientable. A contact structure is cooriented if a choice of global contact form has been made. \\

\noindent Let $(M, \xi)$ be a cooriented contact manifold with fixed global contact form $\alpha$, i.e. $\alpha\in\Omega^1(M)$ satisfies $\ker \alpha= \xi$, $\alpha\wedge d\alpha^n\neq0$. A smooth submanifold $S \stackrel{e}{\hookrightarrow} M$ is called a \emph{contact submanifold} if the induced distribution $\xi_S=e^*(\xi)$ is a contact structure on $S$.\\

\noindent The notion of a blow--up  has its origins in algebraic geometry. First, we define the concept for a
complex vector space. See ~\cite{Ha} for further details.
\begin{definition}\label{def:lnbu}
The blow--up $\widetilde{\C}_0^{n+1}$ of the $n$--dimensional complex vector space $\C^{n+1}$ at the origin is the smooth manifold $\O(-1)=\{([l],p):p\in l\}\subset\CP^n\times\C^{n+1}$ along with the restriction of the projection onto the second factor $\sigma:\O(-1)\longrightarrow\C^{n+1}$.
\end{definition}

\noindent Note that $\sigma$ restricted to $\O(-1)\setminus \{([l],p):p\in l,p\neq0\}$ induces a diffeomorphism onto the image $\C^{n+1}\setminus\{0\}$. The projective space $\sigma^{-1}(\{0\})=\CP^n$ is called the exceptional divisor. The topological blow--up of $M$ along $S$ defined in the previous Section coincides with the previous definition if $S=\{0\}$ is the origin in $M=\C^{n+1}$. More generally, from the definition of $\widetilde{M}$ we can conclude the following
\begin{lemma}\label{prop:tpbl}
Let $M$ be a smooth manifold and $(S,e)$ a submanifold with complex normal bundle. There exists a smooth submanifold $E_S\subset\widetilde{M}$ diffeomorphic to the total space of a projective smooth bundle over $S$ such that, as smooth manifolds, $M\setminus S\cong\widetilde{M}\setminus E_S$.
\end{lemma}

\noindent  The topological blow--up can be performed along any complex submanifold $S$ of a complex
manifold  $M$. In this case the blown--up manifold $\widetilde{M}$ inherits a canonical  complex structure. Analogously, if $(M,\omega)$ is a symplectic manifold and $S$ a symplectic submanifold, the topological blow--up manifold $\widetilde{M}$ can also be endowed with a symplectic structure. In the symplectic case there is no uniqueness, see \cite{MS}.
The topological blow--up could also be performed along a contact submanifold of a contact manifold because the normal bundle is symplectic and hence it is also complex.

\begin{remark} \label{rem:key}
1. Suppose the normal bundle $\nu_M(S)$ splits as a direct sum of isomorphic complex line bundles $L$:
$\nu_M(S)=L\oplus \stackrel{r)}{\cdots} \oplus L$. Then there is a second projection map $\pi_2: \nu_M(S) \longrightarrow \CP^{r-1}$ defined as follows. Given a point $p\in S$, let $s_p\in L_p$ be non--zero vector in the fiber. Then a point $[l_1, \cdots, l_r] \in \nu_M(S)_p$ is mapped to $\pi_2([l_1, \cdots, l_r])=[l_1/s_p, \cdots, l_r/s_p]$. It is simple to verify that the map is well--defined, i.e. independent of the choice of vector $s_p$.\\

\noindent 2. The hypothesis above is satisfied in some cases. For instance, let $S$ be the base locus of a projective, resp. symplectic, Lefschetz pencil. Then $S$ conforms the hypothesis for $r=2$. In such case the fibers of $\pi_2$ are projective, resp. symplectic. This also occurs with contact pencils, see \cite{Pr1}.
\end{remark}

\noindent{\bf Example}: {\it Let $(M^5, \xi)$ be a $5$--dimensional contact manifold and $S$ is a $1$--dimensional contact compact submanifold, i.e. a transverse embedded loop. If we perform a topological blow--up along $S$, the exceptional divisor is $E\cong\S^1 \times \CP^1\cong\S^1\times\S^2$. We are in the hypothesis of the previous Remark: $\nu(S^1)$ is trivial. Therefore we have a projection $\pi_2:\S^1 \times \S^2 \to \S^2$. In the contact case, if we assume that $E$ is a contact submanifold, it is not possible to ensure that the fibers of such projection map are contact:  there is no contact distribution on $\S^1\times\S^2$ whose fibers are all transverse to the contact structure, see ~\cite{Gi}.} \\

\noindent In the previous example, the non--transversality of the fibers occurs only because we are using the topological blow--up as our blown--up manifold. We will further argue from different perspectives that the blown--up manifold $\widetilde{M}$ we should consider in contact topology is not the topological blow--up discussed above. Instead, the correct manifold is obtained through a procedure that substitutes $S\cong\S^1$ by the standard contact sphere $\S^3$, not $\S^1\times\S^2$. In such a case, the natural projection map $\pi:\S^3 \longrightarrow \CP^1$ is the Hopf fibration, whose fibers are transverse to the contact structure.\\



\section{Boothby--Wang Constructions}\label{sec:bw}
In this section we explain the construction of a contact manifold from an integral symplectic manifold as developed in ~\cite{BW}. It will be used in understanding the contact structure on the manifold obtained after a contact blow-up.\\

\noindent A symplectic manifold $(W,\omega)$ is called integral if the class $[\omega]$ lies in the image of the map $H^2(M,\Z) \longrightarrow H^2(M,\R)$, i.e. the periods of $\omega$ are integers. Such a form $\omega$ is called integral. For instance, a K\"ahler form on a complex compact manifold is integral if and only if the manifold is a smooth projective algebraic variety. In the definition above a circle has length $1$. Note that the lift of $[\omega]$  to $H^2(M,\R)$ may not be unique if $H^2(M,\Z)$ contains torsion elements.\\

\noindent Given an integral form $\omega\in H^2(M,\R)$ there exists a Hermitian complex line bundle $L_{\omega}$ admitting a  compatible connection whose curvature is $-i\omega$. See ~\cite{BT} for the details. This leads to the following

\begin{definition} \label{def:bw}
Let $(W,\omega)$ be an integral symplectic manifold. The Boothby--Wang manifold $\S_k(W)$ is the contact manifold whose total space is the unit circle bundle associated to the line bundle $L_{k\omega}$ and its contact structure is defined as the restriction of any connection with curvature $-ik\omega$ to the circle bundle. 
\end{definition}

\begin{remark}\label{rek:Uniqueness_BW}
The contact structure is independent of the choice of connection. Indeed, the space of choices for a connection as above is an affine space modelled on the vector space of flat connections and hence is contractible. The stability theorem of J. Gray applies to ensure the uniqueness up to contactomorphisms  of the contact structures.
\end{remark}

\noindent For the case $k=1$ we will sometimes omit the subindex. Note that the topology of the total space varies with the parameter $k$. The exact relationship
between the topology  and the parameter $k$  is the content of the following
\begin{lemma} \label{lem:cov}
Let $(W,\omega)$ be a symplectic manifold. Then the Boothby--Wang manifold $\S_1(W)$ is a $k$--covering of $\S_k(W)$.
\end{lemma}
\begin{proof}
We fix a Hermitian connection on $L$, this induces a Hermitian connection on $L^{\otimes k}$. Define the unitary non--linear map between line bundles
$$L\longrightarrow L^{\otimes k},\quad u\longmapsto u^{\otimes k}.$$
It preserves the connections on the two bundles. There exists a unitary connection--preserving action of $\Z_k$, the cyclic group of order $k$, in $L$ given by
$$\Z_k\times L\longrightarrow L,\quad (c;u)\longmapsto e^{2\pi ic/k}u.$$
This action induces the trivial action in $L^{\otimes k}$ and thus becomes the deck transformation group of a covering between the total spaces of the associated
circle bundles. This map is certainly compatible with the contact structures.
\end{proof}

\noindent{\bf Examples}: 1. Let $L(k;1,\ldots,1)$ be a lens space, i.e. the orbit space of the action
$$\Z_k\times\S^{2n-1}\longrightarrow\S^{2n-1},\quad 1\cdot(z_1,\ldots,z_n)=(e^{2\pi i/k}z_1,e^{2\pi i/k} z_2, \ldots,e^{2\pi i/k} z_n).$$
The lens space  naturally inherits a contact structure $\xi_{L}$ from the standard contact structure of $\S^{2n-1}$ induced by the complex tangencies. Lemma \ref{lem:cov} provides a contactomorphism between $\S_k(\CP^{n-1})$ and $(L(k;1,\ldots,1), \xi_{L}).$\\

\noindent 2. Consider the $2$--torus $T^2=\S^1\times\S^1$ and $\tau$ an integral area form with total area one. Then the Boothy--Wang manifolds $\S_k(T^2)$ associated to $(T^2,\tau)$ give rise to quotients of the Heisenberg group by discrete subgroups $\Gamma_k$ and thus provide several examples of contact nilmanifolds different from the 3--torus.\\

\noindent The construction of the contact blow--up will involve the quotient of the product of two Boothby--Wang manifolds. With this in mind, we proceed to describe the Boothby--Wang construction when the base sympletic manifold is a product. 
We show that the Boothby--Wang construction and the Cartesian product {\it commute}. In precise terms, let $\S_{(b,a)}(W_1\times W_2)$ be the Boothby--Wang manifold associated to
$$(W_1\times W_2,b\pi_1^*\omega_1+a\pi_2^*\omega_2),$$
then we have the following:
\begin{theorem}\label{thm:bw}
Let $(W_1,\omega_1)$ and $(W_2,\omega_2)$ be symplectic manifolds and $a,b\in\Z$ a pair of coprime integers. Consider the product $\S(W_1)\times\S(W_2)$ of the Boothby--Wang manifolds and the action
$$\xymatrix{
\varphi_{(a,-b)}: \S^1\times\S(W_1)\times\S(W_2)\longrightarrow \S(W_1)\times\S(W_2)\\
(p,q) \longmapsto \theta\cdot(p,q)=(a\theta\cdot p,-b\theta\cdot q)
}$$
Then the space of orbits is a manifold diffeomorphic to $\S_{(b,a)}(W_1\times W_2)$. This space of orbits carries a contact structure induced by a connection with curvature
$$b\pi_1^*\omega_1+a\pi_2^*\omega_2.$$

\noindent and hence is contactomorphic to $S_{(b,a)}(W_1 \times W_2).$

\end{theorem}

\begin{proof} Let $G=\S^1\times\S^1$ and $H\cong\S^1\subset G$ be the subgroup defined as the image of the embedding
$$\varphi_{(a,-b)}:\S^1\longrightarrow H\subset G,\quad \sigma\longmapsto (a\sigma,-b\sigma).$$
Let $P$ be the $G$--principal bundle with base space $W_1\times W_2$ induced by the $\S^1$--principal bundles $\S_1(W_1)$ and $\S_1(W_2)$. Our aim is to describe $P/H$ as a bundle over $W_1\times W_2$. In general $P\longrightarrow P/H$ is not a $H$--principal bundle but it is the case when both $G$ and $H$ are closed Lie groups and $H$ is
a normal sub-group of $G.$ Actually, they are abelian and since $(a,b)=1$, $P/H$ is also a $G/H$--principal bundle over $W_1\times W_2$. Taking into account the exact group sequence
$$1\longrightarrow \S^1\cong H \longrightarrow G\longrightarrow G/H\cong\S^1\longrightarrow 1$$
where the second morphism is given by multiplication by $(b,a)$, we conclude that the space of orbits $P/H$ is a manifold diffeomorphic to $\S_{(b,a)}(W_1\times W_2)$. The claim about the  connection and the associated curvature follows from the short exact sequence
$$0\longrightarrow \Z \stackrel{(a,-b)}{\longrightarrow} \Z\oplus\Z\stackrel{(b,a)^t}{\longrightarrow} \Z\longrightarrow 0.$$
\noindent Finally, it follows from Remark~\ref{rek:Uniqueness_BW} that the manifolds are, in fact,  contactomorphic.

\end{proof}

\noindent There are a few simple cases worth mentioning.\\

\noindent {\bf Examples}: 1. Let $W_1=\{pt.\}$ and $W_2$ arbitrary. Then neither the topology of the resulting space nor the contact structure depend on $b$. Indeed, $\S^1\times\S_1(W_2)/\sim$ is diffeomorphic to
$$\S_{(b,a)}(pt.\times W_2)\cong\S_a(W_2).$$
Analogously, the parameter $a$ is vacuous if $W_2=\{pt.\}$. In particular, $\S^1\times\S^1$ quotiented by any $(a,-b)$ coprime $\S^1$--action is diffeomorphic to $\S^1$.\\

\noindent 2. Let $W_1=W_2=\CP^1$ be symplectic manifolds with the Fubini--Study form. Then the space $\S_{b,a}(\CP^1\times\CP^1)$ is diffeomorphic to $\S^3\times \S^2$ regardless of the values $a,b\in\Z^+$, see ~\cite{WZ} for a proof of this fact. Further, the symplectic structure of the associated line bundle depends only on $a-b$. Note that there is an alternative construction of a contact structure in $\S^3\times \S^2$ using an open book decomposition with $T^*\S^2$ pages and an even power of a Dehn twist as monodromy, however such a procedure may only produce vanishing first Chern class and is thus different from $\S_{b,a}(\CP^1\times\CP^1)$ if $a\neq1$. See ~\cite{Ko} for more.\\

\noindent 3. The previous example can be generalized to construct contact structures on $\S^{2n+1}\times\S^2$. Indeed the result implies that the total space of $\S_{(1,k)}(\CP^n\times\CP^1)$ is a $\S^{2n+1}$--bundle over $\S^2$. The Hopf action is explicit enough for the classifying map to be described as the element $$(n+1)k\in\Z_2\cong\pi_1(SO(2n+2)).$$
Consequently the resulting manifold is diffeomorphic $\S^{2n+1}\times\S^2$ if $n$ is odd or $k$ is even.

\begin{remark}
We would like to remark that it is not known whether the product of any contact manifold with the $2$--sphere admits a contact structure.
\end{remark}

\noindent It will be essential for the contact blow--up construction to be able to extend a connection on a submanifold to a global connection, let us now prove that this is possible under suitable conditions:

\begin{lemma} \label{lem:exten_bw}
Let $S$ be a closed submanifold of $(M^{2n}, \omega)$, possibly with smooth boundary, and $L$ the line bundle associated to $\omega$. Assume that the restriction morphism $H^1(M) \longrightarrow H^1(S)$ is surjective and let $A_S$ be a connection over $L_{|S}$ whose curvature is $-i\omega$. Then, there is a connection $A$ on $L$ with curvature $-i\omega$ such that its restriction to $S$ is $A_S$.
\end{lemma}
\begin{proof}
Let $A_0$ be a connection on the line bundle $L\longrightarrow M$ with curvature $-i\omega$. Denote
$i:S \longrightarrow M$, then $A_S - i^*A_0= \beta_S$ is a closed $1$--form
over $S$. In order to complete our argument we need to extend $\beta_S$ to a global closed 
$1$--form.\\

\noindent By hypothesis the map $H^1(M) \longrightarrow H^1 (S)$ is a surjection. Therefore there exists a
cohomology class $[\beta]$ on $H^1(M)$, such
that restricted to $S$ coincides with $[\beta_S]$. Its difference over $S$ will be the trivial class on $H^1(S)$, so $
\beta_S-i^* \beta= dH_S$, for some smooth function $H_S:S \longrightarrow \R$. We extend
$H_S$ to a global smooth function $H:M \longrightarrow \R$.
The form $A_0+\beta+dH$ is the required global connection with curvature $-i\omega$ and extending $A_S$.
\end{proof}
\section{Surgery along transverse loops}\label{sec:srgy}
Let $(M^{2n+1},\xi)$ be a contact manifold. In this section we recall the blow--up construction from Section 5 in ~\cite{CPP}. It is an operation defined in
a neighborhood of a transversely embedded loop. Topologically it consists of a surgery along the loop: the interior of $\S^1\times B^{2n}$ is removed and a tubular neighbourhood of the $(2n-1)$--sphere $B^2\times\S^{2n-1}$ is glued along the common boundary $\S^1\times \S^{2n-1}$. The sphere $\{ 0 \} \times \S^{2n-1} $ whose neighbourhood is attached is called the exceptional divisor. Let us discuss this surgery operation in the contact category.\\

\noindent Consider the manifold $T=\S^1\times(0,1)\times\S^{2n-1}$ with spherical coordinates $(\theta,r,\sigma)$.
Let $\alpha_{std}=(dr\circ i)|_{\S^{2n-1}}$ be the standard contact form for the contact structure
$$\xi=T\S^{2n-1}\cap i(T\S^{2n-1})$$
on the sphere $\S^{2n-1}\subset\C^n$. Define the following two contact forms in $T$:
\begin{equation} \eta=d\theta-r^2\alpha_{std},\quad \lambda=r^2d\theta+\alpha_{std} \label{eq:std}.
\end{equation}



Fix an integer $l\in\Z$ and consider the diffeomorphism
\begin{equation}
\begin{array}{cccc}
\phi_l:& T & \longrightarrow & T\\   \label{eq:change}
& (\theta,r,z) & \longrightarrow & (\theta,r,e^{2\pi i l\theta}z)
\end{array}
\end{equation}
It pulls--back the contact form $\eta$ to $\overline{\lambda}=(-r^2)\cdot [(l-r^{-2})d\theta+\alpha_{std}]$.\\

\noindent Given a subset $C\subset M$, let $\U(C)$ denote a small neighbourhood of $C$ in $M$. These ingredients suffice to prove the following:


\begin{theorem} \label{thm:blow-up}
$($Thm. 5.1 in ~\cite{CPP}$)$ Let $(M^{2n+1}, \xi)$ be a contact manifold. Let $S\subset M$ be a smooth transverse loop in $M$. There exists a manifold $\overline{M}$  satisfying the following conditions:
\begin{itemize}
\item[-] There exists a contact structure $\overline{\xi}$ on $\overline{M}$.
\item[-] There exists a codimension--$2$ contact submanifold $E$ inside $\overline{M}$ with trivial normal bundle. The manifold $(E,\overline{\xi})$ is contactomorphic to the standard contact sphere $(\S^{2n-1},\xi)$.
\item[-] The manifolds $(M\setminus \U(S), \xi)$ and $(\overline{M}\setminus E, \overline{\xi})$ are contactomorphic.
\end{itemize}
The manifold $(\overline{M},\overline{\xi})$ will be called the contact surgery blow--up of $M$ along $S$. The contact submanifold $(E,\overline{\xi})$ is called the exceptional divisor.
\end{theorem}

\begin{proof} By Gray's stability, we may assume that a tubular neighbourhood of the embedded loop is contactomorphic to $\S^1 \times B^{2n}(\varepsilon)$ with the contact form $\eta$ as in (\ref{eq:std}), for some small radius $r\leq \varepsilon$. We enlarge this tubular $\varepsilon$--neighbourhood using the squeezing technique from ~\cite{EKP} to obtain a radius 2 neighbourhood. More precisely, we need the following auxiliary lemma:
\begin{lemma} \label{lem:EKP} {\em (}Proposition 1.24 in \cite{EKP}{\em )} \label{lem:shrink}
Let $k>0$ be a positive integer and $R_0>0$ a radius. Then the following map is a contactomorphism
\begin{eqnarray*}
\psi_k: \S^1 \times B^{2n}(R_0) & \longrightarrow & \S^1 \times B^{2n}\left( \frac{R_0}{\sqrt{1+kR_0^2}}\right) \\
(\theta, r, w_1, \ldots, w_n) & \longrightarrow & \left(\theta,  \frac{r}{\sqrt{1+kr^2}}, e^{2\pi ik\theta}w_1, \ldots, e^{2\pi ik\theta}w_n\right),
\end{eqnarray*}
and it restricts to the identity at $\S^1 \times \{ 0 \}$.
\end{lemma}

\noindent Consider  $R_0=2$ in the lemma above, then we need $k$ large enough to satisfy
$$\frac{2}{\sqrt{1+4k}}<\varepsilon.$$

\noindent We may therefore  assume that the tubular neighbourhood for which the standard equation (\ref{eq:std}) holds for $\eta$ has radius $r=2$.
In the annulus corresponding to radius $(3/2,2)$ use $\phi_1$ to induce the contact structure given by $\ker \overline{\lambda}$. Declare $\ker \lambda$ to define the contact structure in the radius area $[0,1/2]$. It is left to find a strictly increasing function interpolating between $r^2$ and $1-r^{-2}$ in the middle region. This can be done, see Figure \ref{fig:h}.
\end{proof}
\begin{center}
\begin{figure}[ht]
\includegraphics[scale=0.3]{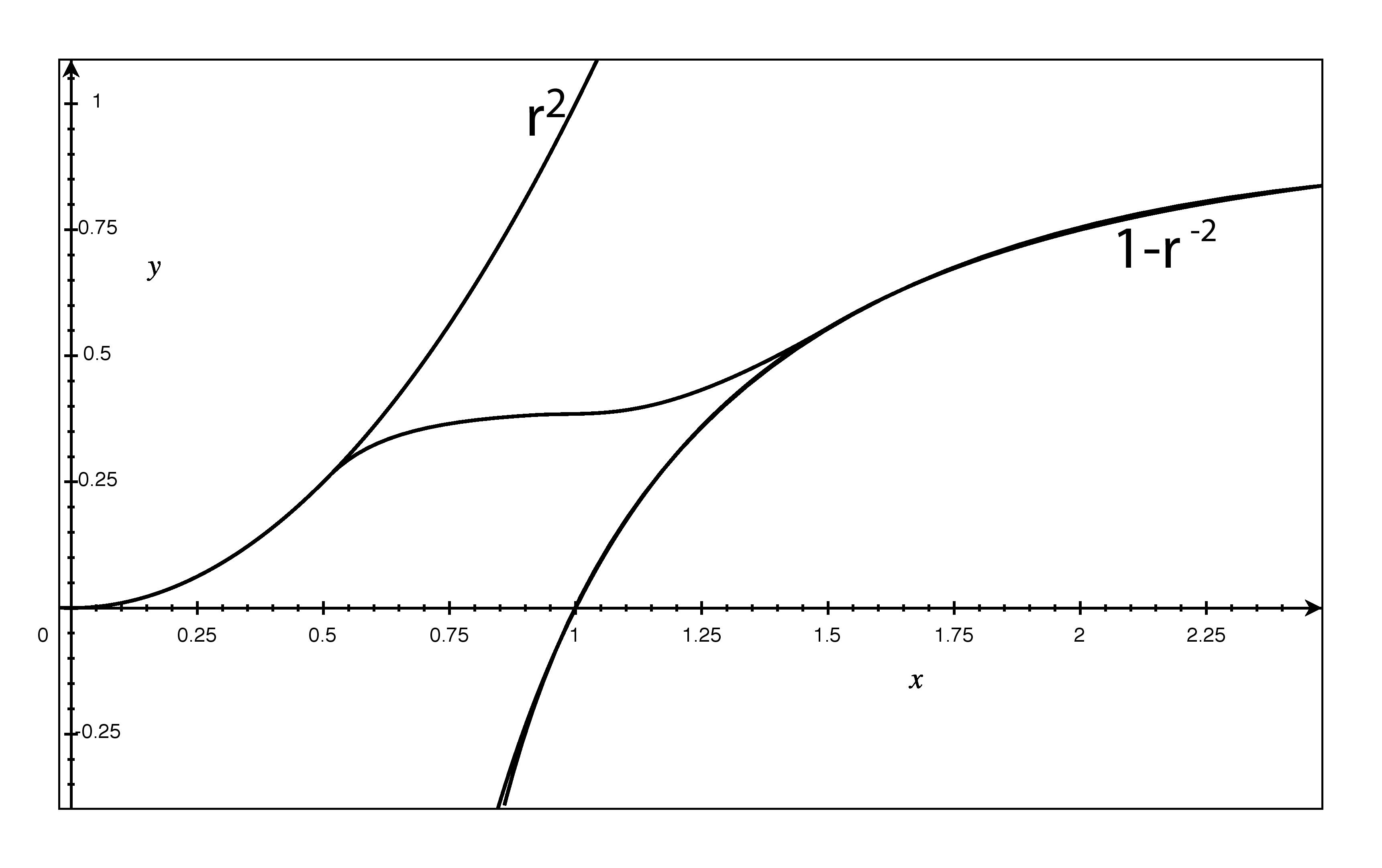}
\caption{Interpolation matching $\lambda$ and $\overline{\lambda}$.}\label{fig:h}
\end{figure}
\end{center}
\begin{remark}

\noindent The process described in the proof can be modified to include the radius squeezing in the gluing map. It suffices to use $\phi_l$ as gluing map instead of $\phi_1$ in the domain. Indeed, denote $T_\rho=\S^1\times(0,\rho)\times\S^{2n-1}$ and consider the contact structures
$$\xi_0=\ker\{d\theta-r^2\alpha_{std}\},\quad\xi_l=\ker\{(l-r^{-2})d\theta+\alpha_{std}\}.$$
Define the map
$$\varphi: T_2\longmapsto T_{\varepsilon(k)},\quad (\theta,r,z)\longmapsto\left(\theta,\frac{r}{\sqrt{1+kr^2}},z\right),$$
where $\varepsilon(k)$ is the obvious radius in the image. Then the following diagram is commutative in the contact category :
$$\xymatrix{
(T_2,\xi_1)\ar@{->}[d]^{\varphi}\ar@{->}[r]^{\phi_1}  & (T_2,\xi_0)\ar@{->}[d]^{\psi_k}\\
(T_{\varepsilon(k)},\xi_l)\ar@{->}[r]^{\phi_l} &  (T_{\varepsilon(k)},\xi_0)
}$$
where Lemma \ref{lem:EKP} is performed with parameter $k=l-1$.

\end{remark}
Note that the contactomorphism type of the exceptional divisor is that of the standard sphere, the parameter in the construction allows us to discretely vary the radius of the tubular neighbourhood we are collapsing. 
\begin{lemma}
The maps $\phi_l$ and $\phi_k$ are smoothly isotopic if and only if $(k-l)n$ is even.
\end{lemma}
\begin{proof}
Let $t\in\S^1$ be a circle coordinate. Consider the morphism
$$\Psi: \pi_1(SO(2n)) \longrightarrow\pi_0(\Diff(\S^1 \times \S^{2n-1})),\quad \Psi(\gamma_t)(\theta,z)= (\theta, \gamma_{\theta}(z)).$$
If $\gamma_k$ denotes $k$--times the standard circle action on $\S^{2n-1}\subset\C^n$ induced by $\C^*$, then it is clear that $\phi_k$ is realized as $\Psi(\gamma_k)$. Since $\pi_1(SO(2n)) \simeq \Z_2$ and $\gamma_k\simeq k \cdot n$ under this identification, $\gamma_k= \gamma_l$ if and only if $(k-l)n$ is even.\\

\noindent It is left to prove that $\phi_0$ and $\phi_1$ are not isotopic, for $n$ even. Construct two manifolds $X_0$ and $X_1$ by gluing two copies of the manifold $B^2 \times \S^{2n-1}$ respectively using $\phi_0$ and $\phi_1$ along the boundary. These manifolds are not diffeomorphic. A sphere is a spin manifold and the product formula for characteristic classes implies that so is $X_0= \S^2 \times \S^{2n-1}$.\\

\noindent The manifold $X_1$ is not spin. This can be seen by using any section $s$ of the twisted bundle $X_1\longrightarrow\S^2$, such $s$ exists because $n\geq 2$. Denote by $\nu(s(\S^2))$ the normal bundle to the section and let $E_1 \longrightarrow \S^2$ be the complex bundle over $\S^2$ such that $\S(E_1)=X_1$. Then $s^*(\nu(s(\S^2)\oplus \R)=E_1$. Note that $w(E_1)=1$ if $n$ is odd and
$$w_2([s(\S^2)])=w_2(TX_{|s(\S^2)})=w_2(\nu(s(\S^2))=w_2(s^*(\nu(s(\S^2)\oplus \R))= w(E_1).$$
Hence $\phi_0$ and $\phi_1$ are not isotopic.
\end{proof}
\noindent In particular, the smooth type of the contact blow--up manifold will depend on the parity of the positive integer fixed for the construction. As for the contact type, it follows from Theorem 1.2 in \cite{EKP} that the maps $\phi_k$ and $\phi_l$ are not contact compactly supported isotopic if $k\neq l$. This does not imply that the contact structures are different, but at least there is no local contactomorphism relating the two contact structures.

\section{Gromov's approach} \label{sec:Gromov}
In this section we develop the contact blow--up along a Boothby--Wang submanifold, as suggested in  ~\cite{Gr}.The
existence of a minimum radius for the tubular neighbourhood of the submanifold along which we will perform the blow--up will play an important role. This feature will be revisited in the definition provided in  Section \ref{sec:quotient}.\\

\noindent Let us review the definition of the symplectic blow--up, see \cite{MS} for more details.

\subsection{Symplectic blow--up}
Let $(M, \omega)$ be a symplectic manifold and $S$ a symplectic submanifold of codimension $k\geq4$. 
Consider the symplectic normal bundle $(\nu_S,\pi)$ of $S$ in $M$ and fix a compatible almost complex structure. The choice of a compatible almost complex structure for a symplectic form induces a metric and the equality
$$U(k)=O(2k)\cap\ Sp(2k,\R)$$
implies that the structure group of $\nu_S$ can be considered to be $U(k)$. Thus $\nu_S$ is an associated vector bundle of a $U(k)$--principal bundle $P \longrightarrow S$.\\


\noindent The symplectic blow--up of $M$ along $S$ is obtained by the fiberwise symplectic blow--up of $\nu_S$. Hence we require Definition \ref{def:lnbu} for the case of symplectic vector spaces. Let $(\R^{2k},\omega_0)$ be the standard symplectic vector space and $\omega_{FS}$ the standard Fubini--Study form on complex projective space. We will use the following

\begin{definition}\label{def:sybu}
A symplectic blow--up of $(\R^{2k},\omega_0)$ at the origin with radius $\delta$ is a symplectic manifold $(\widetilde{\R}^{2k}_{\delta}, \widetilde{\omega}_{\delta})$ such that:
\begin{itemize}
\item[1.] $\widetilde{\R}^{2k}_{\delta}\stackrel{\pi}{\longrightarrow} \R^{2k}$ is a topological blow--up of $\C^{k}$ at the origin. The symplectic form induced on the exceptional divisor $E\cong\CP^{k-1}$ is $\delta^{\frac{1}{2k}} \cdot \omega_{FS}$.
\item[2.] For any $\varepsilon>0$, there exists a symplectomorphism
$$\widetilde{\R}^{2k}_{\delta} \setminus \pi^{-1}(B(\delta+\varepsilon))\cong\R^{2k} \setminus B(\delta+\varepsilon)$$
\item[3.] The unitary group $U(k)$ acts Hamiltonially in  $(\widetilde{\R}^{2k}_{\delta}, \widetilde{\omega}_{\delta})$.\\
\end{itemize}
\end{definition}
\noindent The symplectic blow--up of $(\R^{2k},\omega_0)$ at the origin exist for each $\delta$.
\begin{remark}
Note that the definition depends on $\delta$, this parameter does not appear in Definition \ref{def:lnbu} since any linear homothety at the origin is a complex isomorphism.
\end{remark}

\noindent Let us describe the non--linear symplectic blow--up of $M$ along $S$. Property 3 in the above definition allows us to associate to $P$ a bundle $(\widetilde{\nu}_{S, \delta},\widetilde{\pi})$ over $S$ with fiber $\widetilde{\R}^{2k}_{\delta}$. Let $\beta$ be a connection in $P$, there are induced coupling forms $\alpha$ and $\widetilde{\alpha}_{\delta}$, in $\nu_S$ and $\tilde{\nu}_{S, \delta}$ respectively, restricting to the symplectic form on each fiber and coinciding away from the radius $\delta +\varepsilon$, see Thm. 6.17 in \cite{MS}. Define the forms
\begin{eqnarray*}
\omega_{\nu} & = & \alpha + \pi^* \omega_S \label{eq:la1} \\
\widetilde{\omega}_{\nu} & = & \widetilde{\alpha}_{\delta} + \widetilde{\pi}^* \omega_S \label{eq:la2}
\end{eqnarray*}
on the bundles $\nu_S$ and $\widetilde{\nu}_{S, \delta}$. These are symplectic forms close to the zero section and to the exceptional divisor respectively.\\

\noindent These forms also coincide away from a neighbourhood of $S$ of radius $\delta+\varepsilon$. Let $U_{\delta_0}= P \times_{U(k)} B(\delta_0)$ be a neighbourhood of the zero section of the symplectic normal bundle. By the symplectic neighbourhood theorem there is a neighbourhood $\U(S)$ of the symplectic submanifold $S$ and a symplectomorphism $\Psi: \U(S)\cong\ U_{\delta_0}$. Thus any fiberwise symplectic blow--up on $\nu_S$ with radius $0 \leq \delta+\varepsilon < \delta_0$ can be glued back to the initial manifold $M$ using the symplectomorphism $\Psi$. The resulting manifold is the symplectic blow--up of $M$ along $S$ with radius $\delta$.\\



\noindent Observe that the radius of the tubular neighbourhood of $S$ cannot  be estimated a priori. Therefore the symplectic volume of the exceptional divisor cannot be assumed to be arbitrarily large. This will be an obstruction to develop the Gromov's approach in the contact category.\\

\noindent{\bf Example:} Let $V$ be a rank--$2k$ symplectic vector bundle over a symplectic manifold $(W,\omega)$. Then the total space is symplectic as well. Thus, we are able to blow--up the symplectic manifold $V$ along its zero section $W$. In case the symplectic form $\omega$ is of integer class, the symplectic form in the resulting blown--up manifold will be of an integer class if the blow--up radius is $m^{\frac{1}{2k}}$, $m\in\N^*$. We call this a radius $m$ blow--up. \\

\subsection{Definition of Contact Blow--up}

We now define the contact blow--up in terms of the symplectic blow--up. This is the second notion listed in Section \ref{sec:intro}.\\

\noindent Let $(M,\xi)$ be a contact manifold and $(S,\xi_S)$ a contact submanifold. We assume:
\begin{itemize}
\item[H1.] The contact submanifold $S$ is contactomorphic to a Boothy--Wang manifold $\mathbb{S}(W,\omega)$.\\
\item[H2.] Let $\pi:\mathbb{S}(W)\longrightarrow W$ be the circle bundle projection. There exists a symplectic bundle $V$ over $W$ such that, as symplectic bundles $\nu_M(S)\cong V$.
\end{itemize}



\noindent The total space of $V$ carries a symplectic form $\overline{\omega}$ in the same cohomology class of $[\omega]$, under the natural identification of $H^2(V,\mathbb{R})$ with $H^2(W, \mathbb{R})$. As previously explained, there exists a symplectic manifold $(\widetilde{V},\overline{\omega}_W)$ obtained by blowing up $V$ along its zero section $W$. Suppose that the parameter multiplying the class of the exceptional divisor $E$ in the symplectic blow--up is a positive integer, i.e. the symplectic form in $\widetilde{V}$ is integral.\\

\noindent The construction of the contact blow--up is based on the following diagram:
$$\xymatrix{
&\nu_M(S)\cong\pi^*(V) \ar@{->}[d]
&\mathbb{S}(V) \ar@{->}[d]
&\mathbb{S}(\widetilde{V})\supset\mathbb{S}(E)=\mathbb{S}(\widetilde{V})|_E \ar@{->}[d]
\\
&(S,\xi_S)\cong\mathbb{S}(W) \ar@{->}_{\pi}[d]
&(V,\overline{\omega}) \ar@{->}[ld]
&(\widetilde{V},{\overline{\omega}_W})\supset E \ar@{->}[l]
\\
&(W,\omega)}$$
\begin{center}
\small{Diagram 1. Contact Blow--up Setup} 
\end{center}
\noindent Each map is a bundle projection. It is essential to understand the relation between the contact manifolds $\S(W),\S(V)$ and $\S(E)$. This is the content of the following:

\begin{lemma}
In the hypotheses above, $\mathbb{S}(W)$ is a contact submanifold of $\mathbb{S}(V)$. There are contactomorphic neighbourhoods $\U(\mathbb{S}(W))$ and $\U(S)$
in $\mathbb{S}(V)$ and $M$ respectively. 
\end{lemma}

\begin{proof}
The choice of symplectic form on $V$ implies that there exists a symplectic embedding of $W$ in $V$ and therefore $\mathbb{S}(W)$ is contained in $\mathbb{S}(V)$ as a contact submanifold. The tubular neighbourhood theorem states that the normal bundle $\nu_M(S)$ is diffeomorphic to a small neighbourhood of $S$ in $M$, but $\nu_M(S)\cong\pi^*(V)$ so the same situation applies to $\mathbb{S}(W)$ in $\mathbb{S}(V)$. The statements now follow from the contact neighbourhood theorem.\end{proof}


\noindent In consequence, $\mathbb{S}(W)\subset\mathbb{S}(V)$ provides a local model. Thus we only need to perform the blow--up of $V$ along $W$ and study whether the Boothby--Wang structures associated to them allow us to glue back the resulting blown--up model to $M$. This is the content of the following:\\

\begin{proposition} \label{propo:bundle_blow}
Let $S=\S(W)$ be a Boothby--Wang contact submanifold of $\S(V)$. Suppose we symplectically blow--up $W \subset V$ by collapsing a radius $1$ neighbourhood. Then, there is a choice of contact form for $\mathbb{S}(\widetilde{V})$ such that $\mathbb{S}(E)$ is a contact submanifold of $\mathbb{S}(\widetilde{V})$ and the complement of an arbitrary small neighbourhood of $\mathbb{S}(E)$ in $\mathbb{S}(\widetilde{V})$ is contactomorphic to the complement of some  neighbourhood of $\mathbb{S}(W)$ in $\mathbb{S}(V)$.
\end{proposition}

\noindent For the sake of a clearer exposition the proof is explained at the end of this subsection.\\

\noindent Suppose we can choose a tubular neighbourhood $\U(\mathbb{S}(W))\subset\mathbb{S}(V)$ with radius larger than $1$ which is contactomorphic to a tubular neighbourhood $\U(S)\subset M$. Then we have the following

\begin{definition}
The contact blow--up of $(M,\xi)$ along $(S,\xi_S)$ is the contact manifold $(M',\xi')$ obtained by removing the neighbourhood $\U(S)$ and gluing along its boundary a small neighbourhood of $\mathbb{S}(E)$ in $\mathbb{S}(\widetilde{V})$.
\end{definition}


\noindent The contact manifold $(M',\xi')$ is contactomorphic to $M$ away from small neighbourhoods of $\mathbb{S}(E)$ and $S$ respectively. The \emph{exceptional divisor} of the contact blow--up is defined to be $\mathbb{S}(E)$, where $E$ is the exceptional divisor of the symplectic blow--up over which it is locally modelled. Observe that for the definition to work we need $S$ to have a tubular neighbourhood of radius at least $1$ inside $M$. \\

\noindent{\bf Example}: 1. The most simple example of contact blow--up is the case of a transverse
loop $K$ in $(M^5, \xi)$. The loop is contactomorphic to $\mathbb{S}(pt)$ and its normal bundle is the pull--back of the
trivial bundle over the point. Thus H1 and H2 are satisfied. The symplectic model corresponds to the blow--up of $\C^2$ at the origin, collapsing a neighbourhood of radius $1$, and therefore $E=\CP^1$. Hence, $\mathbb{S}(E)=\mathbb{S}(\CP^1)$, i.e. the standard contact $3$--sphere. This particular case can be seen, at least topologically, as a surgery along a loop.\\

\noindent 2. In the previous example we may symplectically blow--up with radius $k\in \N^*$. The exceptional divisor is then $\S(\CP^1, k\omega_{\CP^1})$, i.e. the sphere bundle associated to the polarization $\SO (k)$ of $\CP^1$, which is the lens space $L(k;1)$ with its standard contact structure. Therefore, even the diffeomorphism type of the blown--up contact manifold changes with the blow--up radius $k\in \N^*$.\\

\noindent Note that there is not natural projection map from $\mathbb{S}(E)$, the exceptional
divisor, to the blow--up locus $\mathbb{S}(W)$. In the case of a loop in a
$5$--dimensional manifold, the exceptional divisor for a radius $1$ blow--up is $\S^3$ and the blow--up locus is
the circle $\S^1$. This is a difference with respect to the symplectic and algebraic cases where the
exceptional divisor is a bundle over the submanifold along which the blow--up is performed. It is true though that there is a natural projection $\mathbb{S}(E) \longrightarrow E \longrightarrow W$, but it does not lift to $\mathbb{S}(W)$.

\begin{remark}
The assumption of the integer radius can be fulfilled in certain cases. For instance in the blow--up along a transverse $\S^1$ we can use the Lemma \ref{lem:shrink}. Therefore the construction in this case will have two natural parameters: the integer radius that determines the topology of the exceptional divisor, and the choice of framing in the spirit of the lemma. The above described construction \`a la Gromov does not show in general the appearance of this second positive integer, this is a reason to introduce a third way of defining the blow--up highlighting these two choices.
\end{remark}

\noindent To conclude this subsection we prove the assertion that allowed us to glue the Boothby--Wang construction over the exceptional divisor in the contact blow--up construction.\\

\noindent{\it Proof of Proposition \ref{propo:bundle_blow}}. We need to find an
appropriate connection on the topological Boothby--Wang manifold corresponding to $\widetilde{V}$.\\


\noindent Notice from the construction of the symplectic blow--up as  given in \cite{MS}, we know  that given a sufficiently small neighbourhood of $E$ in $\widetilde{V}$ it
is possible to choose a symplectic form $\overline{\omega}$ on $\widetilde{V}$ such that
complement of that neighbourhood in $\widetilde{V}$ is symplectomorphic to a small neighbourhood of 
$W$ in $V$.  Furthermore, observe that the exceptional divisor is just the inverse image
of $W$ contained in $V$ as the zero section under the blow--up projection $\phi:\widetilde{V}\longrightarrow V$.\\

\noindent Now recall from the Definition \ref{def:bw} that the contact structure of
$\mathbb{S}(\widetilde{V})$ is determined by the choice of a connection over the associated
line bundle whose curvature is $-i\overline{\omega}$. So let $A$ be the connection over $L$ that determines the contact structure on $\mathbb{S}(V)$, and denote by $U$ an arbitrarily small neighbourhood of $W$ inside $V$. From the construction of the symplectic form $\overline{\omega}$ on $\widetilde{V}$ we can assume that the map $\phi$ is a symplectomorphism of
$V \setminus U$ and $\phi^{-1}(V \setminus U)$. Therefore the connection $\phi^*(A)$ satisfies the required 
properties over $\phi^{-1}(V \setminus U)$. It remains for it to be extended to a
connection all over $\widetilde{V}$ with curvature $-i\overline{\omega}$. By Lemma~\ref{lem:exten_bw} such an extension is possible 
provided that the restriction morphism from $H^1(\phi^{-1}(V \setminus U), \R)\longrightarrow H^1(\widetilde{V}, \R)$ is surjective. It is sufficient to show that
$\pi_1(\widetilde{V}) = \pi_1(\phi^{-1}(V \setminus U))$.\\

\noindent Indeed, observe that $\widetilde{V}$ is homotopic to a $(\CP^{r-1})$--bundle over $W$, with $r \geq 2$, and hence
 $\pi_1(\widetilde{V}) = \pi_1 (W)$ holds. 
Note that the manifold $\phi^{-1}(V \setminus U)$ is diffeomorphic to $V \setminus U$ and the set $V \setminus U$ is homotopy equivalent to a sphere bundle over $W$ with fibers of dimension greater than $2$. From the long exact sequence of homotopy we conclude that 
$$\pi_1(\phi^{-1}(V \setminus U))\cong\pi_1(W).$$
It now follows from Lemma \ref{lem:exten_bw} that
on $\mathbb{S}(\widetilde{V})$ there is a choice of contact form with the required properties. Away from a given small neighbourhood of $\mathbb{S}(E)$ in $\mathbb{S}(\widetilde{V})$ a contact form can be chosen such that it induces a contactomorphism from the complement of such a neighbourhood to the complement of some neighbourhood of $W$ in $V,$ this is because we can choose the symplectic form on $\widetilde{V}$ with the required property. \hfill $\Box$



\section{Blow--up as a quotient} \label{sec:quotient}

In this section we define the contact blow--up of a contact manifold $M$ along a Boothby--Wang contact submanifold $S$ using the notion of contact cuts. 


\subsection{Contact cuts}

Given a  $\S^1$--action on a manifold $M$, topologically the cut construction is based on collapsing the boundary of a tubular neighborhood of a given submanifold invariant by the action. Basic knowledge on the contact reduction procedure is assumed in the next few paragraphs, see ~\cite{Ge}. Let us recall the construction of a contact cut for a contact $\S^1$--action as developed by E. Lerman:
\begin{theorem}\label{thm:contcut} $($Thm. 2.11 in ~\cite{Le1}$)$ Let $(M,\alpha)$ be a contact manifold with a $\S^1$--action preserving $\alpha$ and let
$\mu$ denote its moment map. Suppose that $\S^1$ acts freely on the zero level
set $\mu^{-1}(0)$. Then the set\footnote{The equivalence relation is defined as $m\sim m'\Longrightarrow \mu(m)=\mu(m')=0$ and $m=\theta\cdot m'$ for some $\theta\in\S^1$.}
$$M_{[0,\infty)}:=\{m\in M|\mu(m)\in[0,\infty)\}/\sim$$
is naturally a contact manifold. Moreover, the natural
embedding of the reduced space
$$M_0 := \mu^{-1}(0)/\S^1$$
into $M_{[0,\infty)}$ is contact and the complement
$M_{[0,\infty)}\backslash M_0$ is contactomorphic to the open subset
$$\{m \in M | \mu(m) > 0\}\subset(M,\alpha).$$
\end{theorem}

\begin{remark}
Note that the contact reduction requires the regular value to be $0$, whereas in the symplectic reduction any regular value is licit. This is so because in the contact reduction it is imposed that the orbits of the isotropy subgroup are tangent to the contact structure, see ~\cite{Ge2}.
\end{remark}

\subsection{Blow--up procedure}\label{subsec:bl}
Let $(M^{2n+1},\xi)$ be a contact manifold and $(S,\ker{\alpha})$ a codimension--$2k$ contact submanifold. Suppose that $(S,\ker\alpha)\cong\S_a(W)$ for some symplectic manifold $(W,\omega)$, $a\in\Z^+$, and $\nu_S\cong\underline{\C}^k$ as complex bundles over $S$. We will define the {\it contact blow--up of $M$ along $S$}.

\begin{remark}
Any isocontact embedding\footnote{The embedding $e:(M_1,\xi_1) \longrightarrow (M_2, \xi_2)$ is isocontact if $e^*(\xi_2) = \xi_1$.} of a contact $3$--fold in a sphere has trivial normal bundle. This situation does occur: any closed cooriented $3$--fold admits an isocontact embedding into the standard contact $7$--sphere.
\end{remark}
\noindent A tubular neighbourhood of the contact submanifold $S$ is contactomorphic to
$$S_R=S\times B^{2k}(R)\stackrel{sph. coord.}{\twoheadleftarrow} S\times[0,R)\times \S^{2k-1},\quad\mbox{for some }R\in\R^+,$$
with the contact structure given by $\alpha+r^2\alpha_{std}$, where $\alpha_{std}$ is the standard contact form in $\S^{2k-1}$. Let $b\in\Z^+$ and consider the $\S^1$--action
$$\varphi_{(a,-b)}:\S^1\times S\times [0,R)\times\S^{2k-1}\longrightarrow \S(W)\times [0,R)\times\S^{2k-1}$$
$$(\theta,p,r,z)\longmapsto ((a\theta)\cdot p,r,e^{-2\pi ib\theta}z)$$
This action is generated by the field $X=aR_S-bR_{std}$ where $R_S,R_{std}$ are the Reeb vector fields associated to $\alpha$ and $\alpha_{std}$. 

\noindent The moment map of the above action is
$$\mu_{(a,b)}:S\times B^{2k}(R)\longrightarrow\mathfrak{g}^*\cong \R$$
$$(p,r,z)\longmapsto a-br^2$$
\noindent The contact cut can only be performed in the pre--image of the regular value $0\in\R$, it is thus a necessary condition that $R^2\geq a/b$. This can always be achieved if $b$ is large enough.
\begin{definition}
Let $S\cong\S_a(W)$ be a contact submanifold of $(M,\xi)$ with fixed trivial normal bundle $S\times B^{2k}(R)$.
Let $b\in\Z^+$ be such that $R^2\geq a/b$. The $(a,b)$--contact blow--up $\widetilde{M}_S$ of $M$ along $S$ is defined to be the contact cut of $M$ for the moment map associated to the circle action $\varphi_{(a,-b)}:$
$$\widetilde{M}_S:=M_{\{\mu_{(a,b)}\leq 0\}}$$
\end{definition}
\noindent The collapsed region $\mu_{(a,b)}^{-1}(0)/\sim$ will be called the {\it exceptional divisor}, it is a contact manifold of dimension $2n-1$. The induced $\S^1$--action in the level set
$$\mu_{(a,b)}^{-1}(0)\cong S\times\{\sqrt{a/b}\}\times\S^{2k-1}$$
coincides with the action $\varphi_{(a,-b)}$ defined in Theorem \ref{thm:bw} with $W_1=W$ and $W_2=\CP^{k-1}$. Thus, the orbit space is
$$\mu_{(a,b)}^{-1}(0)/\S^1\cong\S_{(b,a)}(W\times\CP^{k-1})\cong\S(W)\times\S(\CP^{k-1})/\sim.$$

\begin{remark}
Notice that both the topology and the contact structure of the exceptional divisor strongly depend on the choice of the parameters $a$ and $b$. Consequently, so does $\widetilde{M}_S$.
\end{remark}

\noindent{\bf Example}: 1. In the case of a contact $5$--fold, a transverse circle --the simplest contact submanifold-- is replaced by a (quotient of a) standard contact $3$--sphere, as in Section \ref{sec:srgy}. This new construction of the blow--up along a transverse loop will be compared with the previous ones in the next section.\\

\noindent 2. Consider the contact blow--up along a contact $3$--sphere $\S^3\cong\S(\CP^1)\subset M^{2n+1}$. The topology of the exceptional divisor will depend on the element of the corresponding higher homotopy group, cf. Section \ref{sec:bw}.\\

\noindent 3. If in the previous example $(M,\xi)$ is a $5$--dimensional contact manifold, the exceptional divisor of the $(1,k)$ blow--up is contactomorphic to $\S^3$. In higher dimensions, the exceptional divisor of a $(1,k)$ blow--up along $\S^3$ is diffeomorphic to $\S^2\times\S^{2n-3}$ for $n\geq3$ and $k$ even.

\subsection{Blow--up general normal bundle}

We define the contact blow--up along a contact submanifold with a general normal bundle. The construction will clearly coincide with the previous blow--up in the case of a trivial normal bundle.


\subsubsection{Preliminaries} In smooth topology the smooth structure of a neighbourhood of a submanifold is retained by the normal bundle. The contact geometry nearby a contact submanifold $(S,\xi_S)$ is determined by the normal bundle $\nu_S$ along with a conformally symplectic structure. Such a structure exists because $\nu_S$ can be identified with the symplectic orthogonal $\xi_S^\perp$. The statement of the contact neighbourhood theorem is as follows:
\begin{theorem} $($2.5.15 in ~\cite{Ge}$)$
Let $(S_1,M_1)$ and $(S_2,M_2)$ be contact pairs such that $(S_1,\xi_{S_1})$ is contactomorphic to $(S_2,\xi_{S_1})$. If $\xi_{S_1}^\perp\cong \xi_{S_2}^\perp$ as conformally symplectic bundles, then there exists a contactomorphism between suitable neighbourhoods of $S_1$ and $S_2$.
\end{theorem}

\noindent There exist contact submanifolds with non--trivial normal bundle in a closed contact manifold. Let us provide some examples.\\

\noindent {\bf Examples}: 1. Let $(M,\xi=\ker \alpha)$ be a cooriented contact manifold and $\xi$ itself be non--trivial as an abstract vector bundle. The contact form provides a contact embedding $\alpha:M\longrightarrow\S(T^*M)$ such that the normal bundle of the contact submanifold $M$ is isomorphic to $\xi$.\\

\noindent 2. Let $(V,\omega)$ be an integral symplectic manifold and $W$ a symplectic submanifold with non--trivial normal bundle. For instance, $(W,V)=(\CP^1,\CP^3)$. Then the normal bundle of the contact submanifold $\S(W)$ in $\S(V)$ is also non--trivial.\\

\noindent 3. Let $(M^{2n+1},\xi)$ be a closed cooriented contact manifold. Consider an isocontact embedding
$$(M^{2n+1},\xi)\longrightarrow (\S^{4n+3},\xi_{std}),$$
see ~\cite{Gr} for the existence of such an embedding. Since the tangent bundle of the spheres are stable, it is simple to give sufficient conditions for the normal bundle to be non--trivial, e.g. $M$ not spin.

\begin{remark}
The contact blow--up construction has been used in another context. Given a complex vector bundle $E$ on $M$, the contact submanifold $S\subset M$ is defined as the vanishing set of a section in $H^0(M,E)$. Then $c_1(\nu(S))=PD([S])\neq0$. This occurs for the base locus of contact Lefschetz pencil decompositions of $(M,\xi)$. See ~\cite{CPP}.
\end{remark}




\subsubsection{Definition} In the blow--up construction for the trivial normal bundle case there are two circle actions. The first one exists on the contact submanifold $S$, since it is a Boothby--Wang manifold, and it is extended to a local neighbourhood. The second circle action is the gauge action provided by the complex structure in the conformally symplectic normal bundle. The latter is still available in the non--trivial normal bundle case, the former can {\it a priori} no longer be extended in a neighbourhood.\\

\noindent We hence require a lifting condition for the circle action on $S$: the appropriate set--up is depicted as in the Diagram 1 in Section \ref{sec:Gromov}:
$$\xymatrix{
&\nu_M(S)\cong\pi^*(V) \ar@{->}[d]
&\mathbb{S}(V) \ar@{->}[d]
\\
&(S,\xi_S)\cong\mathbb{S}_a(W) \ar@{->}_{\pi}[d]
&(V,\overline{\omega}) \ar@{->}[ld]
\\
&(W,\omega)}$$

\noindent where $V$ is a symplectic bundle over a symplectic manifold $W$. Assume $a=1$ for simplicity.\\

\begin{lemma}
In the hypotheses above, the circle action provided by the Boothby--Wang structure can be naturally extended to a neighbourhood of $S$.
\end{lemma}
\begin{proof}
Since $W$ is a symplectically embedded submanifold of $V$, $\S(W)$ is a contact submanifold of $\S(V)$. The tubular neighbourhood theorem tells us that the normal bundle $\nu_M(S)$ is diffeomorphic to a small neighbourhood of $S$ in $M$, but after the smooth isomorphism $\nu_M(S)\cong \pi^*(V)$ the same situation applies to $\S(W)$ in $\S(V)$. Since the isomorphism holds at the level of symplectic bundles, the contact tubular neighbourhood theorem ensures that there exists a contactomorphism $\Phi$ between a contact neighbourhood of the zero section in $\nu(S)$ and a contact neighbourhood of $\S(W)$ in $\S(V)$. Consequently, the circle action in $\S(V)$ can be carried along $\Phi$ to a neighbourhood of $S$. \end{proof}

\noindent Let us spell out the moment map of the circle action. We refer to the circle action on the normal bundle induced by its complex structure as the gauge action. This action is the natural $\S^1$--action when working with a contact pair $(S,M)$. Further, the radius coordinate $r\in\R^{\geq0}$ is a global coordinate regardless of the non--triviality of the normal bundle. The remaining action described above will be referred as the Boothby--Wang action. It is the natural action when identifying a neighbourhood of $S$ in $M$ with a neighbourhood of $\S(W)$ in $\S(V)$ via the map $\Phi$ in the proof of the lemma.

\begin{lemma}
The moment map of the $\S^1$--action $\varphi_{(1,-1)}$ is $1-\Phi(r)^2$.
\end{lemma}
\begin{proof}
The moment map of the gauge action is $-r^2$. For the Boothby--Wang one, the circle action realizes the Reeb vector field and thus its moment map is $1$. We express the $r$ coordinate through the contactomorphism $\Phi$ as $\Phi(r)$. Since we are using the action $\varphi_{(1,-1)}$ the statement follows.
\end{proof}
\noindent Recall that the contact cut can be performed if $0$ lies in the image of the moment map.
\begin{remark}
The same argument using a multiple of the gauge action concludes that we may modify the action in order to ensure this: $\Phi$ maps the zero section to $\S(W)$ and thus the values of $1-b^2\Phi(r)^2$ form a decreasing sequence in $b$ that eventually crosses zero, $\R^{\leq0}$ being bounded below.
\end{remark}
\noindent The Boothby--Wang action may as well be arranged to period $a$: the concatenation action is denoted $\varphi_{(a,-b)}$. We are in position to write the

\begin{definition} $($Contact Blow--Up$)$ Let $S\cong\S_a(W)$ be a contact submanifold of $(M,\xi)$. Let $a,b\in\Z^+$ be such that the origin is contained in the image of the moment map $\mu_{(a,b)}$ for the action $\varphi_{(a,-b)}$. The $($a,b$)$--contact blow--up $\widetilde{M}_S$ of $M$ along $S$ is defined to be the contact cut of $M$ for the action $\varphi_{(a,-b)}$, i.e. $\widetilde{M}_S:=M_{\{\mu_{(a,b)}\leq0\}}$.
\end{definition}

\section{Uniqueness for Transverse Loops}\label{sec:un}

In this section we relate the three constructions of the contact blow--up. The construction that can be performed in the most general situation is the one involving the contact cut. It has two degrees of freedom: a pair of positive integers $a$ and $b$. These two parameters relate to previous integers appearing in the first two constructions. Indeed, the parameter $l$ in the contact surgery blow--up corresponds to $b$. For Gromov's construction, the choice of collapsing radius $k \in \Z^+$ gives rise, in the case of transverse loops, to the exceptional divisor $L(k,1)$ and it corresponds to the parameter $a$. It is quite obvious that the diffeomorphism type of the blown--up manifolds is the same regardless of the chosen construction as soon as the parameters coincide as just mentioned. \\

\noindent Let us turn our attention to the contact structure: we restrict ourselves to the case of transverse loops. Denote by $\overline{M}_b$ the surgery contact blow--up defined in Section \ref{sec:srgy} with parameter $b$. The contact blow--up as defined in Section \ref{sec:Gromov} with radius $a$ is denoted by $M'_a$. And $\widetilde{M}_{(a,b)}$ will be the contact--cut blow--up as defined in Section \ref{sec:quotient}, performed with parameters $(a,b)$. Let us show that uniqueness holds in this case, more precisely we prove the following

\begin{theorem} \label{thm:unique}
Let $(M,\xi)$ be a contact manifold. Performing the blow--up along a fixed transverse loop with the three procedures introduced previously, the resulting blown--up manifolds $\overline{M}_1$, $M'_1$ and $\widetilde{M}_{(1,1)}$ endowed with the blown--up contact structures are contactomorphic. Further, given any pair of integers $(a,b)$, the following contactomorphisms hold:
$$\left(\overline{M}_b,\overline{\xi}_b\right)\cong\left(\widetilde{M}_{(1,b)},\widetilde{\xi}_{(1,b)}\right),\quad \left(M'_a,\xi_a'\right)\cong\left(\widetilde{M}_{(a,1)},\widetilde{\xi}_{(a,1)}\right).$$
\end{theorem}

The relation between the different constructions is already hinted in Section \ref{sec:srgy}. Since the exceptional contact divisors coincide and the procedure is of a local nature, i.e. the contact manifold is not altered away from a neighbourhood of the embedded transverse loop, the study should focus on the natural annulus contact fibration. Let us review a few facts.\\

\noindent A contact fibration is a fibration $(M, \xi) \longrightarrow B$ such that the fibers are contact submanifolds. We consider contact fibrations over the disk $f: (V, \xi) \longrightarrow B^2$. The base being contractible, the fibration is trivial and we also assume it to be trivialized. Let us introduce the following
\begin{definition}
Let $(r,\theta)$ be polar coordinates on the disk $B^2$. A trivialized contact fibration over the disk $\pi: F \times B^2 \longrightarrow B^2$ is said to be radial if the contact structure admits the following equation 
\begin{equation}
\ker \alpha_0 = \ker  \{ \alpha_F + H d\theta\}, \label{eq:radial}
\end{equation}
where $H: F \times B^{2} \longrightarrow\R$ is a smooth function such that $H=O(r^2)$.
\end{definition}
\noindent Notice that for the total space of a radial contact fibration to have an induced contact structure it is required that
\begin{equation}
\frac{\partial H}{\partial r} >0,\mbox{ for }r>0. \label{eq:contact}
\end{equation}
It is convenient to extend the previous definition in order to include the general situation, where lens spaces may appear as exceptional contact divisors:
\begin{definition}
A trivialized radial contact fibration $\pi: \S^{2n-1} \times B^2 \longrightarrow B^2$ is $\Z_a$--equivariant if the natural diagonal $\Z_a$--action on the fibration preserves the radial contact structure.
\end{definition}
The action in the fiber sphere $\S^{2n-1}$ is generated by an $\frac{2\pi}{a}$--rotation along the Hopf fiber, whereas the action in the base $B^2$ is the standard $\frac{2\pi}{a}$--rotation in the disk. They preserve respectively the standard contact structure in $\S^{2n-1}$ and the $1$--form $d\theta$ in the disk. Hence, the fibration becomes equivariant if the function $H$ is preserved by the action. \\

\noindent Topologically it is fairly straightforward that the blow--up operations we are performing are tantamount to a priori different fillings of the fibration over an annulus to form a manifold lying over the disk -- this being always considered up to a finite action $\Z_a$, for lens spaces fillings. The transition from $\S^1\times B^{2n}$ to $B^2\times\S^{2n-1}$ can be understood in these terms: both fibrations over the annulus --produced by restricting to $r\in (0.5,1)$-- are filled in the origin with a circle and $\S^{2n-1}$ respectively. In the transverse loop case it will be enough to use the following
%
\begin{lemma} \label{lem:unique}
Let $V$ be a manifold with contact structures $\xi_0$ and $\xi_1$. Assume that there are two smoothly isotopic diffeomorphisms 
$$f_0: V \longrightarrow F \times B^2\mbox{ and }f_1: V \longrightarrow F \times B^2,$$
which are contactomorphisms\footnote{A priori, not necessarily contact isotopic.} for $\xi_0$ and $\xi_1$ respectively. Let the two fibrations be radial contact fibrations with common contact fiber $F$ and satisfying that the diffeomorphism
$$f_1 \circ f_0^{-1}: F \times B^2 \longrightarrow F \times B^2$$
is the identity close to the boundary. Then, the contact structures $\xi_0$ are $\xi_1$ are isotopic. \\

\noindent Further, if the fiber is $F\cong\S^{2n-1}$ and the contact fibrations are $\Z_a$--equivariant, the contact structures are isotopic through $\Z_a$--equivariant contactomorphisms.
\end{lemma}
\begin{proof}
This can be reduced to the setup with a fibration $F \times B^2$ with two different radial contact structures
\begin{eqnarray*}
\alpha_0 & = & \alpha_F + H_0 d\theta, \\
\alpha_1 & = & \alpha_F + H_1 d\theta,
\end{eqnarray*}
such that the Hamiltonians $H_0$ and $H_1$ coincide close to the boundary. In that setting, we just need to construct  a path of functions $H_t: F \times B^2 \longrightarrow \R$ connecting them, relative to the boundary, satisfying the contact equation (\ref{eq:contact}) and the condition $H_t=O(r^2)$. But this is possible since the space of such functions is convex. \\

\noindent The argument still works in the equivariant case: the only sentence to be added is that the space of equivariant Hamiltonians is also convex.
\end{proof}

\noindent Thus, to conclude uniqueness we study the contact topology of the different blow--up constructions and ensure that the lemma applies.\\


\noindent {\it Proof of Theorem \ref{thm:unique}}: Let us describe the common model fibration that underlies the three constructions in this case. Consider a standard contact neighbourhood $\S^1\times (0,2) \times \S^{2n-1}$ of the given fixed loop and the morphism
\begin{eqnarray*}
\phi_{(a,b)}: (\S^1 \times (0,2) \times \S^{2n-1}) & \longrightarrow & \S^1 \times (0,2) \times \S^{2n-1} \\
(\theta, r, z) & \longrightarrow & (a\theta, r, e^{2\pi i b \theta} z).
\end{eqnarray*}
It does generalize the diffeomorphism provided by equation (\ref{eq:change}) that reflects the case $a=1$. If $a$ is greater than $1$, it becomes a $a:1$ covering. The covering transformation is provided by $\Z_a$ acting through:
\begin{eqnarray*}
\Z_a \times (\S^1 \times (0,2) \times \S^{2n-1}) & \longrightarrow & (\S^1 \times (0,2) \times \S^{2n-1}) \\
(l, (\theta,r, p)) & \longrightarrow & \left(\frac{2\pi l}{a} + \theta, r, e^{2\pi i bl /a}p\right),
\end{eqnarray*}
which is free as long as $(a,b)=1$. To understand the change in the contact structure, note that the pull--back of the standard contact form $\eta= d\theta -r^2 \alpha_{std}$ is given by
$$\lambda= \phi^*_{(a,b)} \eta = (-r^2)\cdot [(b-ar^{-2})d\theta+\alpha_{std}].$$
Denote by $R_0=\left(\frac{b}{a}\right)^{1/2}$ the critical radius where the distribution becomes horizontal. In these coordinates, for any fixed small $\varepsilon>0$, the projection onto the first two factors
$$\pi: \S^1\times (R_0+\varepsilon,  2) \times \S^{2n-1} \longrightarrow \S^1\times (R_0+\varepsilon,  2)$$
provides a radial contact fibration on the annulus and since the function $(ar^{-2}-b)$ is strictly positive in $(R_0+\varepsilon, 2)$, it can be extended to the interior of the disk to a $\Z_a$--equivariant radial contact fibration. In order to glue back the model to the manifold we should quotient the equivariant contact fibration by $\Z_a$, this allows us to use the map $\phi_{(a,b)}$ to insert the model back into the manifold. \\

\noindent It is thus left to verify that the three blow--up procedures provide examples of such an extension for particular values of $(a,b)$. Then Lemma \ref{lem:unique} will apply to provide the uniqueness of the constructions. Note that the contact surgery blow--up construction is by definition a radial contact fibration, with $a=1$, as shown in Section \ref{sec:srgy}. Let us study the two remaining cases.\\


\noindent To understand the proof in Section \ref{sec:Gromov}, let us proceed backwards and instead of applying the Boothby--Wang construction, we produce a contact structure and then quotient the resulting contact manifold by the Reeb $\S^1$--action to study whether it is the correct object. Once the coordinate change $\phi_{(a,b)}$ is performed, the Reeb vector field $\partial_{\theta}$ becomes
$$\phi_{(a,b)}^*\left(\partial_\theta\right) = \frac{1}{a}\left(\partial_\theta -bR_{std}\right).$$
This vector field extends to the interior of the disk fibration and so we may quotient the resulting manifold $B^2 \times \S^{2n-1}$. We obtain the blown--up symplectic ball $\widetilde{B}^{2n}$ as its quotient. We can further quotient by the free $\Z_a$--action to obtain a non--trivial fibration over the disk $B^2$. This proves that a suitable choice of connection leads  to an equivariant contact fibration. \\

\noindent There are other choices of connection though. From the principal bundle point of view, a radial contact fibration over the annulus $\S^1 \times (0,2)$ corresponds to a connection on
$$B^2 \times \S^{2n-1} \longrightarrow\widetilde{B}^{2n}.$$
Certainly, after Proposition \ref{propo:bundle_blow} the contact structure is fixed with the choice of a connection. Note that the space of connections is affine and thus, after Gray's stability theorem, the resulting contact structures are contact isotopic for different choices of connections. In conclusion, this second model also provides an extension of the model fibration.\\

\noindent We describe the third procedure also beginning with the resulting contact manifold and giving the pull--back of the action. This contact cut construction is also an equivariant radial contact fibration since the pull--back of the vector field generating the $\S^1$--action, that is $$X=a\partial_\theta-bR_{std},$$
is expressed as $\phi^*_{(a,b)}X= \partial_\theta$ after the coordinate change. There, the contact cut is just an equivariant radial contact fibration, see the proof of Theorem 2.11 in ~\cite{Le1}.
\hfill $\Box$ \\

\begin{remark}
Using Lemma \ref{lem:unique}, we can show that the contact blow--up is unique up to the choice of a trivializing chart of the neighbourhood of the transverse loop. In order to prove the uniqueness of the blow--up along transverse loops, we would need to study the space of isocontact embeddings of the contact manifold $\S^1\times B^{2n}$ in $M$. It is probably false that it is connected, which is the requirement needed to ensure the uniqueness of the blow--up once the parameters $a,b$ are fixed.
\end{remark}

\end{document}